\documentclass{amsart}

\usepackage[T1]{fontenc}
\usepackage[english]{babel}
\usepackage{graphicx}
\usepackage{verbatim}
\usepackage{tabularx}
\usepackage{color}
\usepackage{amssymb, amsthm, amsmath}
\usepackage{setspace}
\usepackage{scrhack}
\usepackage{floatrow}
\usepackage{listings}
\usepackage{dsfont}
\usepackage{mathtools}
\usepackage{relsize}
\usepackage{esdiff}
\usepackage{etoolbox}
\usepackage{comment}
\usepackage{accents}
\usepackage[super]{nth}
\usepackage[export]{adjustbox}
\usepackage{scalerel}
\usepackage[hyphens]{url}
\usepackage{enumitem}
\usepackage{tikz}
\usetikzlibrary{intersections, decorations.markings, arrows.meta}

\newtheorem{theorem}{Theorem}[section]

\newtheorem{proposition}[theorem]{Proposition}
\theoremstyle{definition}
\newtheorem{definition}[theorem]{Definition}
\newtheorem{example}[theorem]{Example}
\newtheorem{remark}[theorem]{Remark}
\newtheorem{corollary}[theorem]{Corollary}

\numberwithin{equation}{section}

\makeatletter
\newcommand{\leqnomode}{\tagsleft@true}
\newcommand{\reqnomode}{\tagsleft@false}
\makeatother

\DeclareMathOperator{\Int}{Int}
\DeclareMathOperator{\Ext}{Ext}
\DeclareMathOperator{\tr}{tr}
\DeclareMathOperator{\Eig}{Eig}
\DeclareMathOperator{\ppp}{p}

\newcommand{\CC}[1][]{
	\ensuremath{\mathcal{C}^#1}
}
\newcommand{\R}[1][]{
	\ensuremath{\mathbb{R}^#1}
}
\newcommand{\C}{
	\ensuremath{\mathbb{C}}
}
\newcommand{\Hol}{
	\ensuremath{\mathcal{O}}
}
\newcommand{\N}{
	\ensuremath{\mathbb{N}}
}

\renewcommand{\bf}{\bfseries}

\begin{document}

\reqnomode

\noindent                                             
\begin{picture}(150,36)                               
\put(5,20){\tiny{Submitted to}}                       
\put(5,7){\textbf{Topology Proceedings}}              
\put(0,0){\framebox(140,34)}                          
\put(2,2){\framebox(136,30)}                          
\end{picture}                                         

\vspace{0.5in}

\title[Local geometry of equilibria in holomorphic flows]
{Local geometry of equilibria and a Poincaré-Bendixson-type theorem for holomorphic flows}

\author{Nicolas Kainz}
\address{Institute of Numerical Mathematics, Ulm University, Germany}
\email{nicolas.kainz@uni-ulm.de}

\author{Dirk Lebiedz}
\address{Institute of Numerical Mathematics, Ulm University, Germany}
\email{dirk.lebiedz@uni-ulm.de}



\thanks{\textit{Acknowledgement:} We thank the referee for helpful suggestions and corrections.}

\begin{abstract}
In this paper, we explore the local geometry of dynamical systems $\dot{x}=F(x)$ with real time parameterization, where $F$ is holomorphic on connected open subsets of $\mathbb{C}\stackrel{\sim}{=}\R{2}$. We describe the geometry of first-order equilibria. For equilibria of higher orders, we establish an equivalent condition for "definite directions", allowing us to reverse the implication in Theorem 2 of Chapter 2.10 in [{\it Differential equations and dynamical systems}, Lawrence Perko (1990)] under the additional condition of holomorphy. This enables the geometric construction of a finite elliptic decomposition. We derive a holomorphic Poincaré-Bendixson-type theorem, leading to the conclusion that bounded non-periodic orbits are always homoclinic or heteroclinic.
\end{abstract}

\maketitle

\section{\bf Introduction}
The qualitative topological description of the phase space of differential equations is an important area of research, cf. \cite{andronov1973qualitativetheory,dumortier2006qualitative,perko2013differential,pruss2010gewohnliche}, resulting in many significant insights into the local geometry of equilibria, as discussed in works like \cite{broughan2003holomorphic,broughan2003structure}. Several of these results, especially regarding Newton flows, cf. \cite{heitel2019newton,lebiedz2020hamiltonian,neuberger2014newtonzeta,schleich2018newtonxi}, are relevant for studying holomorphic and meromorphic flows of vectors fields derived from complex functions pertinent to number theory, such as the Riemann $\zeta$- and $\xi$-function, cf. \cite{broughan2005xi,broughan2004zeta}. Generally, the crucial backbone of global phase space of such functions consists of equilibria and separatrices delimiting the "basins of equilibria". Conjectures have been made to relate the specific location of equilibria and separatrices to asymptotic dynamics "at infinity", cf. \cite{schleich2018newtonxi}. Therefore, the topological and analytical characterization of separatrices and "definite directions" in phase space associated with local properties of equilibria is a highly important issue. While dealing with this topic, it turns out that holomorphic vector fields reveal a very fascinating local-global relation in the phase space linked to their locally conformal structure. This investigation serves as the main focus elaborated across a series of papers, with this one being the first in the series. In the following, we give a short chronological guideline to this first paper.

As already mentioned, this paper deals with the special case of holomorphic vector fields in real time parameterization, i.e. a \textit{holomorphic dynamical system} of the form
\begin{align}\label{eq:planarODE}
    \dot{x}=\frac{\mathrm{d}x}{\mathrm{d}t}=F(x),\quad x\in\Omega
\end{align}
with $t\in\R{}$ and $F\in\Hol(\Omega)$, where $\Omega\subset\C\stackrel{\sim}{=}\R{2}$ is open and connected. At first, we characterize the local geometry of simple equilibria, i.e. equilibria of order $m=1$. We categorize simple equilibria into three distinct cases: centers, nodes, and foci.

Subsequently, we describe the local structure of equilibria of order $m\ge2$ in detail. Specifically, we study the result in \cite[Theorem 2.5]{broughan2003holomorphic} using the theory of local sectors, cf. \cite[Chapter 1.5]{dumortier2006qualitative}. An equivalent condition for definite directions enables the existence of a finite elliptic decomposition. This not only allows for reversing the implication in \cite[2.10, Theorem 2]{perko2013differential} and \cite[\S20, Theorem 64]{andronov1973qualitativetheory}, respectively, but also ensures the correctness of the statement in \cite[Theorem 2.5]{broughan2003holomorphic}. For enhanced comprehension, the proof is supported by a geometric visualization. Additionally, we discuss the possibility of transferring this local structure to the global phase portrait by using two suitable examples.

Building upon these foundations, we formulate and prove a specific Poincaré-Bendixson-type result. More precisely, the holomorphy as an additional condition on the vector field permits a more detailed description of the limit sets of bounded orbits. As a Corollary, we conclude that bounded non-periodic orbits must be either homoclinic or heteroclinic. This crucial result sets the stage for a precise topological and geometrical description of "basins of equilibria" and "global elliptic sectors" in subsequent papers.

\section{\bf Common notations and technical preliminaries}
For $z\in\C{}\stackrel{\sim}{=}\R{2}$ we denote by $\Re(z)$ the real part, by $\Im(z)$ the imaginary part and by $\arg(z)$ the argument of $z$. If $\Omega\subset\R{n}$ is open and $f=(f_1,\ldots,f_m)\in\CC{1}(\Omega;\R{m})$, $n,m\in\N{}$, then we denote the Jacobian matrix of $f$ in $x_0\in\Omega$ by $\mathcal{J}_f(x_0)$. The set $\Hol{}(\Omega)$ is the set of holomorphic functions in $\Omega\subset\C{}$. A domain is defined as an open and connected set.

A trajectory or orbit through $x_0\in\Omega\subset\C{}$ corresponding to \eqref{eq:planarODE} is the maximum phase curve denoted by the set $\Gamma(x_0):=x(I)$, where $x$ is the unique solution of \eqref{eq:planarODE} through $x_0$ and $I=I(x_0)\subset\R{}$ the maximum interval of existence with respect to $x_0$. In this notation, we set $\Gamma_+(x_0):=x(I\cap[0,\infty))$ and $\Gamma_-(x_0):=x(I\cap(-\infty,0])$. Each orbit can be evaluated at a time $t\in I$ by the flow $\Phi(t,x_0):=x(t)$, with $x(0)=x_0$. We have $\Gamma(x_0)=\Phi(I(x_0),x_0)$.

Moreover, we use the term of limit sets, cf. \cite[\S4]{andronov1973qualitativetheory} and \cite[Chapter 8.4]{pruss2010gewohnliche}. If $x_0\in\Omega\subset\C{}$ is an initial value, $\Gamma=\Gamma(x_0)\subset\C{}$ the orbit of \eqref{eq:planarODE} through $x_0$ and $I=I(x_0)$ the maximum interval of existence, then we define the positive (negative) limit set as
\begin{align*}
	w_{+(-)}(\Gamma):=\left\{v\in\C{}\,|\,\exists\,(t_k)_{k\in\N{}}\subset I:t_k\overset{k\to\infty}{\longrightarrow}(-)\infty,\,\Phi(t_k,x_0)\overset{k\to\infty}{\longrightarrow}v\right\}\text{.}
\end{align*}
This set does not depend on the initial value, i.e. $w_\pm(\Gamma(\tilde{x}_0))=w_\pm(\Gamma(x_0))$ for all $\tilde{x}_0\in\Gamma(x_0)$. Additionally, if $I(x_0)$ is bounded from above (below), then $w_{+(-)}(\Gamma(x_0))=\emptyset$. Both limit sets exist only when $I(x_0)=\R{}$.

Furthermore, by the Identity Theorem, we can always assume that the zeros of $F$ in \eqref{eq:planarODE} do not have an accumulation point, i.e. on bounded sets there are finitely many equilibria. Since the case $F\equiv 0$ in \eqref{eq:planarODE} is trivial, we do not consider it in this paper.

We also use the Jordan curve Theorem, cf. \cite[Theorem 63.4]{munkres2000topology}, \cite[Lemma 61.1]{munkres2000topology} and \cite[p. 185, Example 4]{munkres2000topology}, as well as the Theorem of Jordan-Schoenflies, cf. \cite[p. 169]{hatcher2005algebraic} and \cite[pp. 376-377]{munkres2000topology}, without citation throughout this paper. If $\Gamma\subset\C{}$ is a closed Jordan curve, we denote the two connecting components resulting from the Jordan curve Theorem by $\Int(\Gamma)$ (the bounded interior of $\Gamma$) and $\Ext(\Gamma)$ (the unbounded exterior of $\Gamma$). If the closed Jordan curve $\Gamma$ lies in a simply connected domain $\Omega\subset\C{}$, then $\Int(\Gamma)\subset\Omega$. The original proof by Jordan can be found in \cite[pp. 587-594]{jordan1893cours}.

\section{\bf Local geometry of simple equilibria}
In the following, we describe the geometric properties of the flow corresponding to the system \eqref{eq:planarODE} locally near simple equilibria, i.e. equilibria of order $m=1$.
\begin{definition}\label{def:GeometryOfEquilibria} \cite[Chapter 2.10]{perko2013differential}.
    Let $\Omega\subset\C$ be a domain and $F\in\Hol(\Omega)$. A point $a\in\Omega$ is an equilibrium of \eqref{eq:planarODE} if $F(a)=0$, i.e. $a\in F^{-1}(\{0\})$. In particular, an equilibrium is called
	\begin{itemize}
		\item[(i)] a center, if there exists $\delta>0$ such that for all $y\in\mathcal{B}_\delta(a)\setminus\{a\}\subset\Omega$ the orbit $\Gamma(y)$ is a closed curve with $a\in\Int(\Gamma(y))$.
		\item[(ii)] a (an) stable (unstable) focus, if there exists $\delta>0$ such that for all $y\in\mathcal{B}_\delta(a)\setminus\{a\}\subset\Omega$ the solution  through $y$ exists globally to the right (left) and satisfies $|\Phi(t,y)|\to|a|$ and $|\arg(\Phi(t,y)-a)|\to\infty$ for $t\to\infty$ ($t\to-\infty$), in particular, $w_+(\Gamma(y))=\{a\}$ ($w_-(\Gamma(y))=\{a\}$). In this case, we say that $\Gamma(y)$ is a spiral.
		\item[(iii)] a stable (unstable) node, if there exists $\delta>0$ such that for all $y\in\mathcal{B}_\delta(a)\setminus\{a\}\subset\Omega$ the solution through $y$ exists globally to the right (left) and satisfies $|\Phi(t,y)|\to|a|$ and $\arg(\Phi(t,y)-a)\to\theta_0$ for $t\to\infty$ ($t\to-\infty$) with a $\theta_0\in[0,2\pi)$, in particular, $w_+(\Gamma(y))=\{a\}$ ($w_-(\Gamma(y))=\{a\}$). In this case, we say that $\Gamma(y)$ tends to $a$ in the definite direction $\theta_0$, i.e. the "pinned tangent vector" of $\Gamma(y)$ tends to the ray originating in $a$ with angle $\theta_0$.
		\item[(iv)] a saddle, if there exist four trajectories $\Gamma_1,\ldots,\Gamma_4$ with $w_+(\Gamma_1)=w_+(\Gamma_2)=\{a\}$ and $w_-(\Gamma_3)=w_-(\Gamma_4)=\{a\}$ and $\delta>0$ such that for all $y\in\mathcal{B}_\delta(a)\setminus\{a\}\subset\Omega$ there exists a $\tau>0$ with $\Phi(t,y)\not\in\mathcal{B}_\delta(a)$ for all $|t|>\tau$.
	\end{itemize}
\end{definition}
\begin{theorem}\cite[Theorem 2.1]{broughan2003holomorphic}\label{thm:equilibria_characterization}
	Let $\Omega\subset\C$ be a domain and $F\in\Hol(\Omega)$. Let $a\in\Omega$ be an equilibrium of \eqref{eq:planarODE} with $F^\prime(a)=\alpha+\mathrm{i}\beta\not=0$. Then:
	\begin{itemize}
		\item[(i)] $\Eig(J_f(a))=\left\{F^\prime(a),\overline{F^\prime(a)}\right\}$.
		\item[(ii)] If $\alpha\not=0$ and $\beta=0$, $a$ is a node. If $\alpha<0$ ($\alpha>0$), the node is asymptotically stable\footnote{cf. \cite[2.9, Definition 1]{perko2013differential}.} (repelling and unstable).
		\item[(iii)] If $\alpha\not=0$ and $\beta\not=0$, $a$ is a focus. If $\alpha<0$ ($\alpha>0$), the focus is asymptotically stable (repelling and unstable).
		\item[(iv)] If $\alpha=0$ and $\beta\not=0$, $a$ is a center or a focus.
        \item[(v)] The equilibrium $a$ cannot be a saddle.
	\end{itemize}
\end{theorem}
\begin{proof}
    By the Cauchy-Riemann equations, we have
    \begin{align*}
		J:=\mathcal{J}_F(a)=\left(\begin{matrix}
			\alpha&-\beta\\
			\beta&\alpha
		\end{matrix}\right)
    \end{align*}
    and the characteristic polynomial $\ppp_J(\lambda)=(\alpha-\lambda)^2+\beta^2$. As complex conjugated eigenvalues we get $\lambda_{\pm}:=\alpha\pm\mathrm{i}\beta$. Additionally, we have the trace $\tr(J)=2\alpha$, the determinant $\det(J)=\alpha^2+\beta^2$ and $\tr(J)^2-4\det(J)=4\alpha^2-4(\alpha^2+\beta^2)=-4\beta^2\le 0$. Consider the linearized system $x^\prime=Jx$ on $\R{2}$. For this system we can use the results in \cite[Chapter 1.5]{perko2013differential}. If $\alpha\not=0$ and $\beta=0$, there is only one real eigenvalue and $a$ is a node for the linearized system. If $\alpha\not=0$ and $\beta\not=0$, $a$ is a focus. If $\alpha=0$ and $\beta\not=0$, $a$ is a center.
    
    By using \cite[Chapter 2.10, Theorem 4]{perko2013differential}, we conclude the assertions (ii) and (iii). The claims regarding stability follow from the principle of linearized stability, \cite[Satz 5.4.1]{pruss2010gewohnliche}. The vector field $F$ is real analytic in $a$ and thus (iv) follows from \cite[Chapter 2.10, Corollary to Theorem 5]{perko2013differential}, which excludes the existence of a center-focus, cf. \cite[p. 206, Dulac Theorem]{perko2013differential}. Since all cases for $\alpha$ and $\beta$ have already been covered, a cannot be a saddle, i.e. (v) holds.
\end{proof}

\section{\bf Definite directions and elliptic decomposition}
In this section, we characterize which directions in an equilibrium of order $m\ge2$ are definite and whether a finite elliptic decomposition occurs. We proof that the implication in \cite[2.10, Theorem 2]{perko2013differential} and \cite[\S20, Theorem 64]{andronov1973qualitativetheory}, respectively, is even an equivalence under the stronger condition of holomorphy.
\begin{definition}\label{def:sector} \cite[Chapter 1.5]{dumortier2006qualitative}, \cite[Definition 4.14 b)]{masterthesiskainz}.
    Let $\Omega\subset\C{}$ be a domain, $F\in\Hol{}(\Omega)$, $F\not\equiv 0$, $a\in\Omega$ an equilibrium of \eqref{eq:planarODE} and $\mu:[0,1]\to\Omega$ a closed piecewise continuously differentiable Jordan curve. Set $\Gamma:=\mu([0,1])$. Assume that $\overline{\Int(\Gamma)}\subset\Omega$ and $\overline{\Int(\Gamma)}\cap F^{-1}(\{0\})=\{a\}$. Let $\nu:\Gamma\to S^1$ be the outer unit normal of $\Gamma$ defined at all points where $\mu$ is differentiable.
    \begin{itemize}
        \item[a)] For $p,q\in\Gamma$ we denote by $\Gamma(p,q)$ the closed (i.e. including $p$ and $q$) curve section of $\Gamma$ from $p$ to $q$ in counterclockwise direction. If $p=q$, we set $\Gamma(p,q)=\{p\}=\{q\}$.
        \item[b)] A sector $S\subset\overline{\Int(\Gamma)}$ of \eqref{eq:planarODE} in $a$ with respect to $\Gamma$ is a compact set such that there exist two so-called characteristic orbits $\Gamma_1,\Gamma_2$ of \eqref{eq:planarODE} and two intersection points $p_1,p_2\in\Gamma$ with the following properties:
        \begin{itemize}
            \item[(i)] $\Gamma_1\not=\Gamma_2$ and $p_1\not=p_2$.
            \item[(ii)] $\mu$ is differentiable at $p_1$ and $p_2$. For $i\in\{1,2\}$, $\Gamma_i\cap\Gamma=\{p_i\}$ and $\langle F(p_i),\mu(p_i)\rangle\not=0$, i.e. the intersection is non-tangential.
            \item[(iii)] The boundary of $S$ is given by
            \begin{align*}
                \partial S=\big(\left(\Gamma_1\cup\Gamma_2\right)\cap\Int(\Gamma)\big)\cup\Gamma(p_1,p_2)\cup\{a\}\text{.}
            \end{align*}
        \end{itemize}
        \item[c)] Let $S$ be a sector of \eqref{eq:planarODE} in $a$ with respect to $\Gamma$ with characteristic orbits $\Gamma_1,\Gamma_2$, intersection points $p_1,p_2\in\Gamma$ and the curve piece $\Lambda:=\Gamma(p_1,p_2)$. Then $S$ is called an elliptic sector with clockwise (counterclockwise) direction if there exist two points $E_1,E_2\in\Lambda\setminus\{p_1,p_2\}$ such that $\Lambda$ is differentiable everywhere except for the two points $E_1$ and $E_2$ and such that the following properties are satisfied:
        \begin{itemize}
            \item[(i)] $\Gamma(E_1)=\Gamma(E_2)\subset\overline{\Int(\Gamma)}$ with $\Gamma(E_1)\cap\Gamma=\Gamma(E_1,E_2)$.
            \item[(ii)] $w_+(\Gamma(E_1))=w_-(\Gamma(E_1))=\{a\}$, i.e. $\Gamma(E_1)$ is homoclinic.
            \item[(iii)] Set $\Lambda_1:=\Gamma(p_1,E_1)\setminus\{E_1\}$ and $\Lambda_2:=\Gamma(E_2,p_2)\setminus\{E_2\}$. For all $y_1\in\Lambda_1$ and $y_2\in\Lambda_2$:
            \begin{itemize}[leftmargin=6mm]
                \item[$\bullet$] $\langle F(y_1),\nu(y_1)\rangle<(>)\;0$.
                \item[$\bullet$] $\langle F(y_2),\nu(y_2)\rangle>(<)\;0$.
                \item[$\bullet$] $\Gamma_{+(-)}(y_1)\cup\Gamma_{-(+)}(y_2)\subset\overline{\Int(\Gamma)}$.
                \item[$\bullet$] $w_{+(-)}(\Gamma(y_1))=w_{-(+)}(\Gamma(y_2))=\{a\}$.
            \end{itemize}
            \item[(iv)] For all $z\in\tilde{S}:=\Int(\Gamma(E_1)\cup\{a\})$, $\Gamma(z)\subset\overline{\Int(\Gamma)}$ and $w_+(\Gamma(z))=w_-(\Gamma(z))=\{a\}$, where
            \begin{align*}
                \hspace{1cm}S\setminus\tilde{S}=\{a\}\cup\Gamma(E_1)\cup\,\bigcup_{\mathclap{y_1\in\Lambda_1}}\,\Gamma_{+(-)}(y_1)\cup\,\bigcup_{\mathclap{y_2\in\Lambda_2}}\,\Gamma_{-(+)}(y_2)\text{.}
            \end{align*}
        \end{itemize}
        \item[d)] The system \eqref{eq:planarODE} has a finite elliptic decomposition (FED) of order $d\in\N{}\setminus\{1\}$ in $a$, if the equilibrium is not a center, node or focus and if there are $d$ characteristic orbits $\Gamma_1,\ldots,\Gamma_d$ with corresponding points $p_1\ldots,p_d\in\Gamma$ such that the following properties are satisfied for all $i\in\{1,\ldots,d\}$:
		\begin{itemize}
		  \item[(i)] The set $\kappa:=\big(\left(\Gamma_i\cup\Gamma_{i+1}\right)\cap\Int(\Gamma)\big)\cup\Gamma(p_i,p_{i+1})\cup\{a\}$ is a closed piecewise continuously differentiable Jordan curve.
			\item[(ii)] The set $S:=\overline{\Int(\kappa)}$ is an elliptic sector of \eqref{eq:planarODE} in $a$ with respect to $\Gamma$ whose characteristic orbits are $\Gamma_i,\Gamma_{i+1}$ and whose intersection points are $p_i,p_{i+1}\in\Gamma$.
			\item[(iii)] The characteristic orbits with the corresponding points are ordered cyclicly and counterclockwise with respect to $a$. Using the notation in (ii), we set $\Gamma_{d+1}:=\Gamma_1$ and $p_{d+1}:=p_1$.
		\end{itemize}
    \end{itemize}
\end{definition}
\begin{remark}
    Let $\Gamma$ be a homoclinic orbit in an equilibrium $a$, i.e. $w_+(\Gamma)=w_-(\Gamma)=\{a\}$. A parameterization with compact time interval for the closed Jordan curve $\Gamma\cup\{a\}$, as in Definition \ref{def:sector} c) (iv), can be constructed in the following manner:
    
    Fix $z\in\Gamma$ and define the map $\gamma:[-1,1]\to\C{}$ by
    \begin{align*}
        \gamma(t):=\begin{cases}
            \Phi\left(\frac{t}{(1-t)(1+t)},z\right)&\text{if }t\in(-1,1)\\
            a&\text{if }t\in\{-1,1\}
        \end{cases}\text{.}
    \end{align*}
    We have $\gamma(0)=z$. The map $t\mapsto\frac{t}{(1-t)(1+t)}$ is strictly monotonously increasing on $(-1,1)$ and converges to $\pm\infty$ for $t\to\pm 1$. Moreover, the existential quantifier in the definition of limit sets can be replaced by an universal quantifier, if the limit set consists of only one element. Thus, by applying \cite[Theorem 9]{andronov1973qualitativetheory}, we conclude
    \begin{align*}
        \lim_{t\to -1^+}\gamma(t)=\lim_{t\to 1^-}\gamma(t)=a
    \end{align*}
    as well as the continuity of $\gamma$.
\end{remark}

In the proof of \cite[Theorem 2.5]{broughan2003holomorphic} the author uses the equation
\begin{align}\label{eq:condition_H}
    \cos(\theta_0)Q_m(\cos(\theta_0),\sin(\theta_0))-\sin(\theta_0)P_m(\cos(\theta_0),\sin(\theta_0))=0\qquad
\end{align}
in \cite[2.10, Theorem 2]{perko2013differential} and \cite[\S20, Theorem 64]{andronov1973qualitativetheory}, respectively, not only as a necessary, but also as a sufficient condition for definite directions without a detailed justification. The following proposition guarantees the validity of \cite[Theorem 2.5]{broughan2003holomorphic}. Additionally, it ensures that the implications in \cite[2.10, Theorem 2]{perko2013differential} and \cite[\S20, Theorem 64]{andronov1973qualitativetheory} can be reversed as well, if the vector field is holomorphic.

\begin{proposition}[Equivalent condition for definite directions] \label{prop:definiteDirections2}
    Let $\Omega\subset\C{}$ be a domain, $a\in\Omega$, $F=F_1+\mathrm{i}F_2\in\Hol{}(\Omega)$, $F\not\equiv 0$ and $F(a)=0$. Assume that the order of $a$ is $m\in\N{}\setminus\{1\}$. Let $F_1^{[k]}$ and $F_2^{[k]}$ be the sum of all terms of the Taylor series at $a$ of $F_1$ and $F_2$ with degree $k\in\mathbb{N}$. Define the function $H(x_1,x_2):=x_1F_2^{[m]}(x_1,x_2)-x_2F_1^{[m]}(x_1,x_2)$, $(x_1,x_2)\in\R{2}$, and the set
    \begin{align*}
		\mathcal{E}(F,m):=\left\{\frac{\ell\pi-\arg(F^{(m)}(a))}{m-1}\mod 2\pi:\ell\in\mathbb{Z}\right\}\subset[0,2\pi)\text{.}
    \end{align*}
    Then:
    \begin{itemize}
        \item[(i)] $\theta_0\in[0,2\pi)$ fulfills the equation $H(\cos(\theta_0),\sin(\theta_0))=0$ if and only if $\theta_0\in\mathcal{E}(F,m)$. It holds $|\mathcal{E}(F,m)|=2m-2$.
        \item[(ii)] Every orbit tending to $a$ does so in a definite direction $\theta_0\in\mathcal{E}(F,m)$.
        \item[(iii)] For all $\theta_0\in\mathcal{E}(F,m)$ there exist $r,\delta>0$ such that for every $x_0\in A(r,\delta):=\{x\in\R{2}:|x-a|<r,|\arg(x-a)-\theta_0|<\delta\}\subset\Omega$ the orbit though $x_0$ tends to $a$ in the definite direction $\theta_0$. These orbits tend to $a$ for $t\to\infty$ ($t\to-\infty$) if and only if $\lambda(\theta_0):=\cos(\arg(F^{(m)}(a))+\theta_0(m-1))<(>)\;0$.
        \item[(iv)] For all $\theta_0\in\mathcal{E}(F,m)$ with $\lambda(\theta_0)<(>)\;0$ there exists $\hat{r}\in(0,r)$ and $\hat{\delta}\in(0,\delta)$ such that $\Gamma_{+(-)}(x_0)\in A(r,\delta)$ for all $x_0\in A(\hat{r},\hat{\delta})$.
    \end{itemize}
\end{proposition}
\newpage
\begin{proof}
    By considering the shifted vector field $\hat{F}(x):=F(x+a)$, we can assume w.l.o.g. $a=0$. The proof is divided into several steps.\\[2.2mm]
    \textbf{Step 1: Transformation into polar coordinates} \cite[\S20, 1.]{andronov1973qualitativetheory}
    
    Calculating the derivatives $\frac{\mathrm{d}|x|}{\mathrm{d}t}$ and $\frac{\mathrm{d}\arg(x)}{\mathrm{d}t}$ leads to the system
	\begin{equation}\label{eq:planarODE_polcoord1}
		  \begin{split}
		      \rho^\prime(t)&=F_1(\rho\cos(\theta),\rho\sin(\theta))\cos(\theta)+ F_2(\rho\cos(\theta),\rho\sin(\theta))\sin(\theta)\\
                \theta^\prime(t)&=\frac{F_2(\rho\cos(\theta),\rho\sin(\theta))\cos(\theta)-F_1(\rho\cos(\theta),\rho\sin(\theta))\sin(\theta)}{\rho}
		  \end{split}
	\end{equation}
    on the simply connected domain $\Omega_1:=(0,\rho^\star)\times\R{}$ with a small $\rho^\star>0$. The relation between the orbits of \eqref{eq:planarODE} and \eqref{eq:planarODE_polcoord1} is determined in \cite[\S8.3]{andronov1973qualitativetheory}. For $j\in\{1,2\}$, $k\in\mathbb{N}$ and $\theta\in\R{}$ set $\tilde{F}_j^{[k]}(\theta):=F_j^{[k]}(\cos(\theta),\sin(\theta))$. With this notation we have $F_j^{[k]}(\rho\cos(\theta),\rho\sin(\theta))=\rho^k\tilde{F}_j^{[k]}(\theta)$. By applying the time transformation $\mathrm{d}\tau=\rho^{m-1}\mathrm{d}t$ and the real analyticity of $F_1$ and $F_2$ in $a$, there exist two analytic functions $\Psi_1$ and $\Psi_2$ such that the orbits of \eqref{eq:planarODE_polcoord1} coincide with those of the system
    \begin{align}\label{eq:planarODE_polcoord2}
		\begin{split}
			\rho^\prime(\tau)&=\rho\tilde{F}_1^{[m]}(\theta)\cos(\theta)+\rho \tilde{F}_2^{[m]}(\theta)\sin(\theta)+\rho^2\Psi_1(\rho,\theta)\\
			\theta^\prime(\tau)&=\tilde{F}_2^{[m]}(\theta)\cos(\theta)-\tilde{F}_1^{[m]}(\theta)\sin(\theta)+\rho\Psi_2(\rho,\theta)
		\end{split}
    \end{align}
    on the domain $\Omega_1$, cf. \cite[\S1, 7., Lemma 8]{andronov1973qualitativetheory}. This system can also be considered on $\Omega_2:=(-\rho^\star,\rho^\star)\times\R{}$. But only on $\Omega_1$ we can relate the orbits of \eqref{eq:planarODE} with those of \eqref{eq:planarODE_polcoord2}, cf. \cite[\S8.3]{andronov1973qualitativetheory}.\\[2.2mm]
    \textbf{Step 2: Suitable representation of the Taylor expansions}
    
    Set $c_k:=\frac{1}{k!}F^{(k)}(a)$ for all $k\in\N{}$. We have $c_m\not=0$. The complex Taylor expansion of $F$ in $a$ is
    \begin{align*}
		F(z)=\sum\limits_{k=1}^\infty c_kz^k\quad\forall\,z\in\mathcal{B}_{r_0}(a)\subset\Omega
    \end{align*}
    with a sufficiently small radius of convergence $r_0>0$. In particular, reduce $r_0$ such that $\mathcal{B}_{r_0}(a)\cap F^{-1}(\{0\})=\{a\}$. It follows
    \begin{align*}
		F_1(z)=\Re(F(z))=\Re\left(\sum\limits_{k=1}^\infty c_kz^k\right)=\sum\limits_{k=1}^\infty \Re\left(c_kz^k\right)\quad\forall\,z\in\mathcal{B}_{r_0}(a)
    \end{align*}
    and analogous with $F_2=\Im(F)$. By using the Binomial theorem, it is easy to see that $F_1^{[k]}(x,y)=\Re\left(c_k(x+\mathrm{i}y)^k\right)$, i.e. this is already exactly the sum of all addends with degree $k\in\N{}$ in the real Taylor expansion of $F_1$. Analogous we have $F_2^{[k]}(x,y)=\Im\left(c_k(x+\mathrm{i}y)^k\right)$. We conclude for $k\in\N$ and $\theta\in\R{}$ the formulas
    \begin{align*}
        \tilde{F}_1^{[k]}(\theta)=|c_k|\cos(\arg(c_k)+k\theta),\quad\tilde{F}_2^{[k]}(\theta)=|c_k|\sin(\arg(c_k)+k\theta)\text{.}
    \end{align*}
    Define $\beta:=\arg(c_m)=\arg(F^{(m)}(a))$ and $\tilde{H}(\theta):=H(\cos(\theta),\sin(\theta))$ with $\theta\in\R{}$.\\[2.2mm]
    \textbf{Step 3: Calculation of the zeros of $\tilde{H}$}
    
    By using a trigonometric addition formula, we get the concise representation\footnote{This important trick enables a much simpler representation of the condition \eqref{eq:condition_H}, if the vector field is holomorphic.}
    \begin{align*}
        \tilde{H}(\theta)=\cos(\theta)\tilde{F}_2^{[m]}(\theta)-\sin(\theta)\tilde{F}_1^{[m]}(\theta)=|c_m|\sin(\beta+m\theta-\theta)\text{.}
    \end{align*}
    Since $|c_m|\not=0$, we conclude that $\tilde{H}(\theta)=0$ if and only if there exists $\ell\in\mathbb{Z}$ such that $\beta+m\theta-\theta=\ell\pi$, i.e. $\theta=\frac{\ell\pi-\beta}{m-1}$. By using the definition of $\mathcal{E}(F,m)$, it follows $\tilde{H}^{-1}(\{0\})=\mathcal{E}(F,m)$. Moreover, for all $\ell\in\mathbb{Z}$ we have
    \begin{align*}
		\frac{(\ell+2m-2)\pi-\beta}{m-1}=\frac{\ell\pi-\beta}{m-1}+\frac{(2m-2)\pi}{m-1}=\frac{\ell\pi-\beta}{m-1}+2\pi\text{.}
    \end{align*}
    Hence $\mathcal{E}(F,m)$ has exactly $2m-2$ elements and (i) holds.\\[2.2mm]
    \textbf{Step 4: Existence of $r,\hat{r}$ and $\delta,\hat{\delta}$}
    
    Let $\theta_0\in\mathcal{E}(F,m)$, i.e. $\theta_0$ satisfies $\tilde{H}(\theta_0)=0$. Consider the $\CC{\infty}$-system \eqref{eq:planarODE_polcoord2} on $\Omega_2$. The point $a_1:=(0,\theta_0)$ is an equilibrium of \eqref{eq:planarODE_polcoord2}. Denote by $G:\Omega_2\to\R{2}$ the right-hand side of \eqref{eq:planarODE_polcoord2}. We calculate the linearization
    \begin{align*}
		\mathcal{J}_{G}(a_1)=\left(\begin{matrix}
			\tilde{F}_1^{[m]}(\theta_0)\cos(\theta_0)+\tilde{F}_2^{[m]}(\theta_0)\sin(\theta_0)&0\\
			\Psi_2(0,\theta_0)&\tilde{H}^\prime(\theta_0)
		\end{matrix}\right)
    \end{align*}
    and the derivative
    \begin{align*}
        \tilde{H}^\prime(\theta_0)=|c_m|(m-1)\cos(\beta+m\theta_0-\theta_0)\text{.}
    \end{align*}
    Additionally, by using another trigonometric addition formula and step 2, we calculate
    \begin{align*}
        \lambda_1:=\tilde{F}_1^{[m]}(\theta_0)\cos(\theta_0)+\tilde{F}_2^{[m]}(\theta_0)\sin(\theta_0)=|c_m|\cos(\beta+m\theta_0-\theta_0)\text{.}
    \end{align*}
    Hence we find the two real eigenvalues $\lambda_1$ and $\lambda_2:=(m-1)\lambda_1$. We get $\lambda(\theta_0)=\frac{\lambda_1}{|c_m|}$. Since $\tilde{H}(\theta_0)=0$ and  sine and cosine do not have common zeros, we have $\lambda_1\not=0$. Hence, since $m-1>0$ and $|c_m|\not=0$, the linearization $\mathcal{J}_{G}(a_1)$ has two real non-zero eigenvalues with the same sign. Thus, $a_1$ is a stable or unstable node of \eqref{eq:planarODE_polcoord2}, cf. \cite[2.10, Theorem 4]{perko2013differential}. More precisely, $a_1$ is stable as well as attractive if and only if $\lambda_1<0$ and unstable as well as repelling if and only if $\lambda_1>0$, cf. \cite[Chapter 1.5]{perko2013differential}. Assume w.l.o.g $\lambda_1<0$, i.e. $a_1$ is asymptotically stable.
    
    We find $r\in\left(0,\min\{r_0,\rho^\star\}\right)$ and $\delta>0$ such that for all $\xi\in(-r,r)\times(\theta_0-\delta,\theta_0+\delta)\cap\Omega_1$ we have $|\Phi(\tau,\xi)|\to\theta_0$ for $\tau\to\infty$. Since the orbit through a point in $\{0\}\times\R{}$ stays on $\{0\}\times\R{}$, the set $\Omega_1$ is invariant\footnote{cf. \cite[2.5, Definition 2]{perko2013differential}.}. So for $x_0\in A(r,\delta)\subset\Omega$ and the orbit $\Gamma$ through $x_0$ of \eqref{eq:planarODE} we find $\xi\in(-r,r)\times(\theta_0-\delta,\theta_0+\delta)\cap\Omega_1$ and the orbit $\Gamma_1$ through $\xi$ of \eqref{eq:planarODE_polcoord2} which corresponds to $\Gamma$ (cf. Step 1). We get $\Gamma_+(\xi)\subset\Omega_1$ and $w_+(\Gamma_1)=\{a_1\}$. Hence the corresponding orbit $\Gamma$ tends to $a$ in the definite direction $\theta_0$. This relation also implies that $\Gamma$ tends to $a$ for $t\to+(-)\infty$ if and only if $\lambda(\theta_0)<(>)\;0$. This shows (iii).
    
    With the same argumentation, since $a_1$ is also attracting, there exists $\hat{r}\in(0,r)$ and $\hat{\delta}\in(0,\delta)$ such that $\Gamma_+(x_0)\in A(r,\delta)$ for all $x_0\in A(\hat{r},\hat{\delta})$. Note that we can always find a rectangle lying in an arbitrarily small circle and a circle lying in an arbitrarily small rectangle. This proves (iv).\\[2.2mm]
    \textbf{Step 5: Nonexistence of further definite directions}
    
    We have already shown that there are at least $2m-2>0$ definite directions, all given by $\mathcal{E}(F,m)$. It remains to show that we have exactly these $2m-2$ definite directions. Let $\Gamma$ be an arbitrary orbit tending to $a$. We have $m-1\ge1$, i.e. $m-1\not=0$. By choosing $\Theta:=\frac{\pi-2\beta}{2(m-1)}$ and $(\hat{x},\hat{y}):=(\cos(\Theta),\sin(\Theta))$, we have
    \begin{align*}
        H(\hat{x},\hat{y})=\tilde{H}(\Theta)=|c_m|\sin(\beta+(m-1)\Theta)=|c_m|\not=0\text{,}
    \end{align*}
    i.e. $H\not\equiv 0$ and we can apply \cite[\S20, Theorem 64]{andronov1973qualitativetheory}. If $\Gamma$ was a spiral, $a$ would be a focus. But we have already found orbits tending to $a$ in certain definite directions (cf. Step 4). Hence, $\Gamma$ tends to $a$ in a definite direction $\tilde{\theta}_0$ and satisfies $\tilde{H}(\tilde{\theta}_0)=0$. Thus, by Step 3, $\tilde{\theta}_0\in\mathcal{E}(F,m)$ and there are no more definite directions than these in $\mathcal{E}(F,m)$. We summarize that also (ii) holds.
\end{proof}
\begin{theorem}[Existence of elliptic decomposition]\label{thm:elliptic_decomposition}
    Let $\Omega\subset\C{}$ be a domain, $F=F_1+\mathrm{i}F_2\in\Hol{}(\Omega)$, $F\not\equiv 0$ and $a\in\Omega$ an equilibrium of \eqref{eq:planarODE} of order $m\in\N{}\setminus\{1\}$. Then $a$ is not a center, node or focus. Furthermore, the system \eqref{eq:planarODE} has a FED of order $d:=2m-2$ in $a$. Additionally, the characteristic orbits of each sector tend to $a$ in adjacent definite directions given by $\mathcal{E}(F,m)$, i.e. the sectors have pairwise empty intersection up to the characteristic orbits and $a$.
\end{theorem}
\begin{proof}
    Assume w.l.o.g. $a=0$. Since $\Omega$ is open and $F\not\equiv 0$, we can find $r_0>0$ such that $\mathcal{B}_{r_0}(a)\subset\Omega$ and $\mathcal{B}_{r_0}(a)\cap F^{-1}(\{0\})=\{a\}$. By Proposition \ref{prop:definiteDirections2}, the equilibrium has exactly $d$ definite directions $\theta_1,\ldots,\theta_{d}\in\mathcal{E}(F,m)$. Assume that the angles in $\mathcal{E}(F,m)$ are ordered cyclic and counterclockwise with respect to $a$. For every $i\in\{1,\ldots,d\}$ we have $r_i,\delta_i>0$ such that for every $x_0\in A_i:=\{x\in\R{2}:|x-a|<r_i,|\arg(x-a)-\theta_i|<\delta_i\}\subset\Omega$ the orbit though $x_0$ tends to $a$ in the definite direction $\theta_i$. It is geometrically clear that $a$ cannot be a center or focus, cf. Definition \ref{def:GeometryOfEquilibria}. Set $\beta:=\arg(F^{(m)}(0))$ and $\lambda_i:=\cos(\beta+\theta_im-\theta_i)$, $i\in\{1,\ldots,d\}$. Proposition \ref{prop:definiteDirections2} characterizes whether the orbits near the ray with angle $\theta_i$ tend to $a$ for positive or negative time. We calculate
    \begin{align*}
		\lambda_i=\cos(\beta+\theta_i(m-1))= \cos(i\pi)=\begin{cases}
		  1&\text{if }i\text{ even}\\
		  -1&\text{if }i\text{ odd}
		\end{cases}
    \end{align*}
    and conclude ($d\ge 2$) that there is a pair of orbits, one reaching $a$ with positive time and one with negative time. Hence the equilibrium cannot be a node. In addition, the directions are alternating, i.e. every ray has no adjacent ray with the same direction.
    
    Consider the $\CC{\infty}$-system \eqref{eq:planarODE_polcoord2} on $\Omega_2=(-\rho^\star,\rho^\star)\times\R{}$ with $\rho^\star>0$ sufficiently small. Assume that $\rho^\star<\min\{r_0,r_1,\ldots,r_d\}$. In Step 4 in the proof of Proposition \ref{prop:definiteDirections2} we figured out, that the points $b_i:=(0,\theta_i)$, $i\in\{1,\ldots,d\}$, are all nodes, i.e. there exists $\varepsilon_i\in(0,\rho^\star)$ such that for all $y\in\mathcal{B}_{\varepsilon_i}(b_i)$ the orbit $\Gamma(y)$ tends to $b_i$ for positive or negative time, depending on the sign of $\lambda_i$. Set $\varepsilon:=\min\{\varepsilon_1,\ldots,\varepsilon_d\}>0$.
    
    First, fix $i\in\{1,\ldots,d\}$. The nodes $b_i$ and $b_{i+1}$ are connected by the heteroclinic orbit $\Xi_i:=\{0\}\times(\theta_i,\theta_{i+1})$.\footnote{If $i=d$, we assume $\theta_{i+1}=\theta_1$. Note that the phase portrait of \eqref{eq:planarODE_polcoord2} is $2\pi$-periodic along the direction of $\theta$.} Assume w.l.o.g that $b_i$ is repelling and $b_{i+1}$ is attracting. Choose $y\in\{0\}\times\R{}$ and $T>0$ such that $\mathcal{B}_{\frac{\varepsilon}{4}}(y)\subset\mathcal{B}_\varepsilon(b_i)$, $\Phi(T,y)\in\mathcal{B}_\varepsilon(b_{i+1})$ and $\mathcal{B}_{\frac{\varepsilon}{4}}(\Phi(T,y))\subset\mathcal{B}_\varepsilon(b_{i+1})$. By the continuous dependence on initial conditions, cf. \cite[Chapter 2.4, Theorem 4]{perko2013differential}, there exists $\delta\in(0,\frac{\varepsilon}{4})$ such that $|\Phi(t,z)-\Phi(t,y)|<\frac{\varepsilon}{4}$ for all $z\in\mathcal{B}_\delta(y)$ and $t\in[0,T]$. Choose $z\in\mathcal{B}_\delta(y)\cap\Omega_1$. By our choice of $\varepsilon$ and $\delta$, the orbit $\Gamma(z)\subset\Omega_1$ is also a heteroclinic orbit connecting $b_i$ and $b_{i+1}$. Depending on the index $i$, we denote the so constructed orbit $\Gamma(z)$ by $G_i\subset\Omega_1$.
    
    Secondly, by using the definition of $\mathcal{E}(F,m)$ in Proposition \ref{prop:definiteDirections2}, we calculate the distance $s:=|\theta_{i+1}-\theta_i|=\frac{\pi}{m-1}>0$ between the nodes on $\{0\}\times\R{}$. Set $\zeta:=\min\{\varepsilon,\frac{s}{2}\}>0$. By applying the theory of circles without contact around nodes, cf. \cite[\S3, 10.-14.]{andronov1973qualitativetheory} and \cite[\S18, Lemma 3]{andronov1973qualitativetheory}, we find continuously differentiable closed paths $C_i\subset\mathcal{B}_{\zeta}(b_i)\subset\Omega_2$ lying nowhere tangential to the flow of \eqref{eq:planarODE_polcoord2} and satisfying $\Int(C_i)\cap F^{-1}(\{0\})=\{b_i\}$. In particular, from the equations (6) and (11) in \cite[\S7, 1.]{andronov1973qualitativetheory} and the remarks made in \cite[\S7, 2.]{andronov1973qualitativetheory} it follows that the paths $C_i$ can be chosen as linear transformed circles or ellipses. By our choice of $\zeta$, these circles without contact have pairwise empty intersection. Moreover, we can ensure, cf. \cite[\S3, 10., Figure 54]{andronov1973qualitativetheory}, that every orbit tending to $b_i$ crosses $C_i$ exactly once. Thus, for all $i\in\{1,\ldots,d\}$ there exist exactly two points $\hat{E}_{i,1},\hat{E}_{i,2}\in\Omega_1$ satisfying $\{\hat{E}_{i,1}\}=G_i\cap C_i$ and $\{\hat{E}_{i,2}\}=G_i\cap C_{i+1}$. Choose $\hat{p}_i\in C_i(\hat{E}_{i-1,2},\hat{E}_{i,1})\setminus\{\hat{E}_{i-1,2},\hat{E}_{i,1}\}\not=\emptyset$ and define $\hat{\Gamma}_i:=\Gamma(\hat{p}_i)$ as well as the curve pieces $\hat{\Lambda}_{i,1}:=C_i(\hat{p}_i,\hat{E}_{i,1})\subset\Omega_1$ and $\hat{\Lambda}_{i,2}:=C_{i+1}(\hat{E}_{i,2},\hat{p}_{i+1})\subset\Omega_1$. Here we use the same notation as in Definition \ref{def:sector} a). Note that these curve sections do not have to be traversed in counterclockwise direction. Depending on $\lambda_i$, we choose the direction in such a way that $\hat{\Lambda}_{i,1},\hat{\Lambda}_{i,2}\subset\Omega_1$. By construction, the set
    \begin{align*}
        \hat{\Gamma}:=\bigcup\limits_{i=1}^d \left(\{\hat{p}_i\}\cup\hat{\Lambda}_{i,1}\cup G_i(\hat{E}_{i,1},\hat{E}_{i,2})\cup
        \hat{\Lambda}_{i,2}\cup\{\hat{p}_{i+1}\}\right)\subset\Omega_1\text{.}
    \end{align*}
    is a piecewise continuously differentiable path lying completely in $\Omega_1$. This geometric construction is visualized in Figure \ref{fig:geom_visual_elliptic_decomp}.
    
    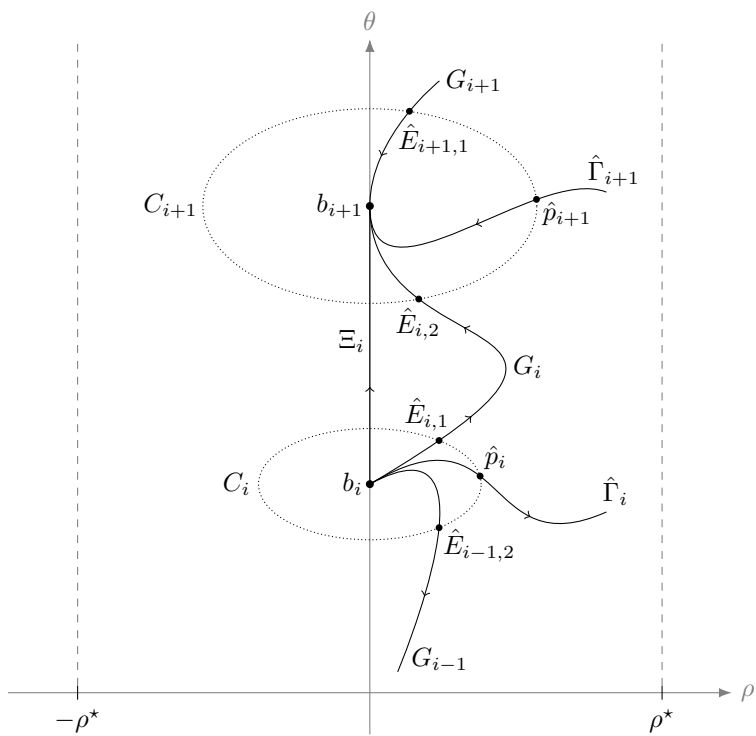
\begin{figure}[ht]
        \begin{center}
            \begin{tikzpicture}[scale=1.85]
                \draw[-{Latex}, gray] (-2.6,0) -- (2.6,0) node[right] {$\rho$};
                \draw[-{Latex}, gray] (0,-0.3) -- (0,4.7) node[above] {$\theta$};
                \draw[dashed, gray] (-2.1,0.05) -- (-2.1,4.7);
                \draw[dashed, gray] (2.1,0.05) -- (2.1,4.7);
                \draw (-2.1,0.05) -- (-2.1,-0.05) node[below] {$-\rho^\star$};
                \draw (2.1,0.05) -- (2.1,-0.05) node[below] {$\rho^\star$};
                \filldraw (0,1.5) circle (0.7pt) node[left, xshift=0.3mm] {$b_i$};
                \filldraw (0,3.5) circle (0.7pt) node[left, xshift=0.4mm] {$b_{i+1}$};
                \draw[densely dotted, black, name path=ellipse1] (0,1.5) ellipse (0.8 and 0.4) node[left, xshift=-14.6mm] {$C_i$};
                \draw[densely dotted, black, name path=ellipse2] (0,3.5) ellipse (1.2 and 0.7) node[left, xshift=-21.6mm] {$C_{i+1}$};
                \draw[black, name path=orbit1, postaction={decorate}, decoration={markings, mark= at position 0.3 with {\arrow{Classical TikZ Rightarrow}}, mark= at position 0.6 with {\arrow{Classical TikZ Rightarrow}}}] (0,1.5) .. controls (2.2,2.8) and (0,2.3) .. (0,3.5) node[right, xshift=17.8mm, yshift=-21.4mm]{$G_i$};
                \draw[black, name path=orbit2, postaction={decorate}, decoration={markings, mark= at position 0.35 with {\arrowreversed{Classical TikZ Rightarrow}}}] (0,3.5) .. controls (0,4) and (0.5,4.4) .. (0.5,4.4) node[right, xshift=-0.4mm, yshift=-0.1mm]{$G_{i+1}$};
                \draw[black, name path=orbit3, postaction={decorate}, decoration={markings, mark= at position 0.7 with {\arrow{Classical TikZ Rightarrow}}}] (0,1.5) .. controls (1,2.05) and (0.2,0.15) .. (0.2,0.15) node[right, xshift=0.6mm, yshift=1.6mm]{$G_{i-1}$};
                \draw[black, name path=orbit4, postaction={decorate}, decoration={markings, mark= at position 0.5 with {\arrowreversed{Classical TikZ Rightarrow}}}] (0,3.5) .. controls (0,2.7) and (1.2,3.8) .. (1.7,3.6) node[right, xshift=-3.7mm, yshift=2.9mm]{$\hat{\Gamma}_{i+1}$};
                \draw[black, name path=orbit5, postaction={decorate}, decoration={markings, mark= at position 0.7 with {\arrow{Classical TikZ Rightarrow}}}] (0,1.5) .. controls (1,2.1) and (0.8,0.9) .. (1.7,1.3) node[right, xshift=-1.7mm, yshift=2.8mm]{$\hat{\Gamma}_i$};
                \draw[postaction={decorate}, decoration={markings, mark= at position 0.35 with {\arrow{Classical TikZ Rightarrow}}}] (0,1.5) -- (0,3.5) node[left, xshift=0.6mm, yshift=-18mm] {$\Xi_i$};
                \fill[name intersections={of=orbit1 and ellipse1, by={intersection}}] (intersection) circle (0.7pt) node[above, xshift=-1.6mm, yshift=0.1mm]{$\hat{E}_{i,1}$};
                \fill[name intersections={of=orbit1 and ellipse2, by={intersection}}] (intersection) circle (0.7pt) node[below, xshift=-0.1mm, yshift=0.1mm]{$\hat{E}_{i,2}$};
                \fill[name intersections={of=orbit2 and ellipse2, by={intersection}}] (intersection) circle (0.7pt) node[below, xshift=3.3mm, yshift=-0.4mm]{$\hat{E}_{i+1,1}$};
                \fill[name intersections={of=orbit3 and ellipse1, by={intersection}}] (intersection) circle (0.7pt) node[right, xshift=-0.7mm, yshift=-2.3mm]{$\hat{E}_{i-1,2}$};
                \fill[name intersections={of=orbit4 and ellipse2, by={intersection}}] (intersection) circle (0.7pt) node[right, xshift=-0.4mm, yshift=-1.9mm]{$\hat{p}_{i+1}$};
                \fill[name intersections={of=orbit5 and ellipse1, by={intersection}}] (intersection) circle (0.7pt) node[above, xshift=2.1mm, yshift=-0.3mm]{$\hat{p}_i$};
             \end{tikzpicture}
        \end{center}
        \caption{Geometrical visualization of the construction in the proof of Theorem \ref{thm:elliptic_decomposition} for the case where $b_i$ is repelling and $b_{i+1}$ is attracting.}
        \label{fig:geom_visual_elliptic_decomp}
    \end{figure}

    By using the relation between the orbits of \eqref{eq:planarODE} with those of \eqref{eq:planarODE_polcoord2} on $\Omega_1$, cf. Step 1 in the proof of Proposition \ref{prop:definiteDirections2}, we get the corresponding points $p_i,E_{i,j}\in\mathcal{B}_{\rho^\star}(a)\subset\Omega$ and the corresponding paths $\Gamma,\Gamma_i,\Lambda_{i,j}\subset\mathcal{B}_{\rho^\star}(a)\subset\Omega$, $i\in\{1,\ldots,d\}$, $j\in\{1,2\}$. By construction, $\Gamma$ is a closed piecewise continuously differentiable Jordan curve differentiable everywhere except for the points $E_{i,j}$. Denote the parameterization of $\Gamma$ by $\mu:[0,1]\to\Omega$. We claim that the curve $\Gamma$ together with the characteristic orbits $\Gamma_i$ and points $p_i$ with $i\in\{1,\ldots,d\}$ form a FED of order $d$ in $a$.

    The properties in (i)-(iii) in Definition \ref{def:sector} c) can be checked directly by using the above results on $\Omega_2$. The above notation coincides with those in Definition \ref{def:sector} c). It remains to show that also property (iv) is fulfilled. Define the closed Jordan curves $J_i:=\{b_i\}\cup \Xi_i\cup\{b_{i+1}\}\cup G_i$, $i\in\{1,\ldots,d\}$. Then the property (iv) can be formulated equivalently on $\Omega_1$ as follows: For all $i\in\{1,\ldots,d\}$ and $x\in\Int(J_i)\subset\Omega_1$ the orbit $\Gamma(x)\subset\Omega_1$ is a heteroclinic orbit connecting $b_i$ and $b_{i+1}$.

    Fix $i\in\{1,\ldots,d\}$ and $x\in\Int(J_i)$. We have $\Gamma(x)\subset\overline{\Int(J_i)}$, which is compact and invariant. Moreover, by our choice of $r_0$ and $\rho^\star$, we have $\Int(J_i)\cap F^{-1}(\{0\})=\emptyset$. Hence, by \cite[\S11, Corollary 2]{andronov1973qualitativetheory}, there are no periodic orbits in $\Int(J_i)$. From the Poincaré-Bendixson Theorem, cf. \cite[3.7, Theorem 1]{perko2013differential}, it follows that the limit sets of $\Gamma(x)$ both consist of at least one equilibrium. By applying the results in \cite[\S4, 3.]{andronov1973qualitativetheory}, both limit sets are closed, connected and invariant. Thus, since the only equilibria in $\overline{\Int(J_i)}$ are the two nodes $b_i$ and $b_{i+1}$ (one repelling and the other attracting), $\Gamma(x)\subset\Int(J_i)$ must be a heteroclinic orbit connecting $b_i$ and $b_{i+1}$.

    All in all, we conclude the existence of a FED in $a$. Moreover, by the choice of the points $\hat{p}_i\in\Omega_1$, $i\in\{1,\ldots,d\}$, the characteristic orbits $\Gamma_i$ of each elliptic sector tend to $a$ in adjacent definite directions.
\end{proof}
\begin{remark}
    We have proved Theorem \ref{thm:elliptic_decomposition} constructively, i.e. we have explicitly constructed the geometrical objects required for a FED. An alternative proof is also possible by applying the Poincaré-Bendixson Index Theorem, cf. \cite[Theorem 2.2]{izydorek1996note}, \cite[Proposition 6.32]{dumortier2006qualitative} and \cite[Appendix, p. 511]{andronov1973qualitativetheory}. This alternative proof version has a few more technical difficulties to solve, since several geometrical possibilities that could a priorily occur must be excluded. A detailed proof can be found in \cite[Chapter 4.3]{masterthesiskainz}.
\end{remark}
\begin{corollary}\label{cor:index_1_center_foci_nodes}
    \it Let $\Omega\subset\C{}$ be a domain, $F\in\Hol{}(\Omega)$ and $F\not\equiv 0$. Then all nodes, centers and foci have order and index\footnote{cf. \cite[Chapter 9.6]{pruss2010gewohnliche} and \cite[Chapter V]{andronov1973qualitativetheory}.} $1$. All equilibria possessing a FED have order and index greater than $1$.
\end{corollary}
\begin{proof}
    By Theorem \ref{thm:elliptic_decomposition}, nodes, centers and foci must have order $1$. By \cite[Theorem 2.4]{broughan2003holomorphic}, these equilibria have also index $1$. The second assertion follows from Definition \ref{def:sector} d) and again \cite[Theorem 2.4]{broughan2003holomorphic}.
\end{proof}
\begin{example}\label{ex:exp}
    Consider $F:\C{}\to\C{}$ with $F(x):=x^5e^x$, $F\in\Hol{}(\C{})$. The unique zero $a=0$ has order $m=5$. By Theorem \ref{thm:elliptic_decomposition}, there exists a FED of order $d=2m-2=8$. We calculate the Taylor series of $F$ by
    \begin{align*}
        F(x)=x^5\sum_{k=0}^{\infty}\frac{x^k}{k!}=\sum_{k=0}^{\infty}\frac{x^{k+5}}{k!}=\sum_{k=5}^{\infty}\frac{1}{(k-5)!}x^k
    \end{align*}
    and conclude the definite directions
    \begin{align*}
        \mathcal{E}(F,5)=\left\{\frac{\ell\pi}{4}:1\le\ell\le d\right\}\text{.}
    \end{align*}
    
    \begin{figure}[ht]
        \centering
        \includegraphics[trim={14pt 1pt 39pt 25pt},clip,scale=0.82]{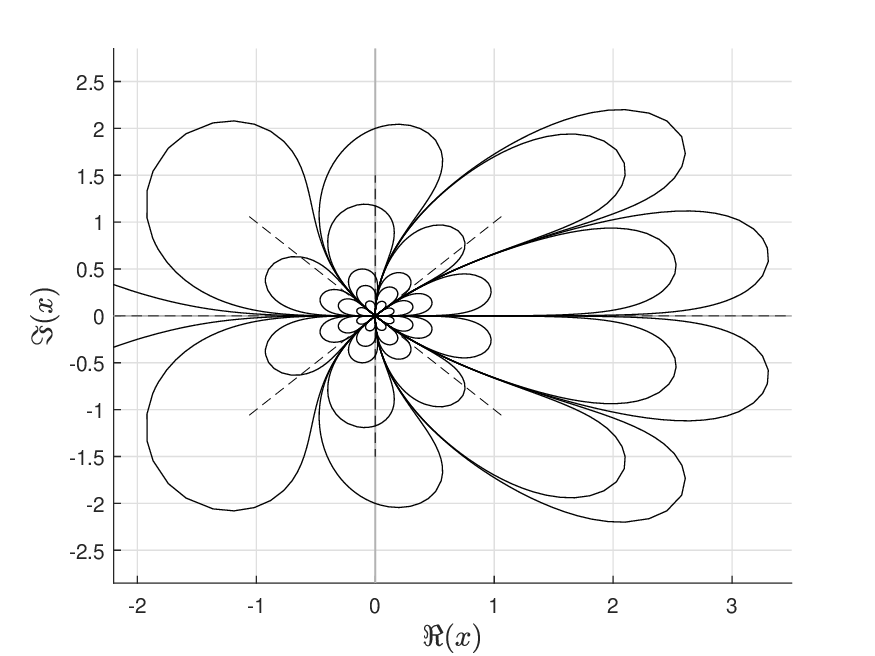}
        \caption{Local phase portrait of system \eqref{eq:planarODE} with $F(x)=x^5e^x$, plotted with Matlab.}
    \end{figure}
    
    We verify that there are indeed $d$ elliptic sectors in $a$. Moreover, we can see that the exponential function "tugs" all orbits to the right. Thus, a priori it is not clear if there is a different topological structure globally near the ray with angle $\theta=\pi$ (the negative $\Re$-axis). It can be shown that infinitely many orbits in $\C{}\setminus\{a\}$ have $a$ in its negative limit set, although they are unbounded. As a direct consequence, these orbits cannot be homoclinic. We conclude that in this example the local elliptic structure described in Definition \ref{def:sector} c) cannot be transferred to the global phase portrait.
\end{example}
\begin{example}
    Consider $F:\C{}\to\C{}$ with $F(x):=x^3(x-1)^3$, $F\in\Hol{}(\C{})$. This is a polynomial of degree $6$ having two zeros $a_1=0$ and $a_2=1$, both of order $m=3$. By Theorem \ref{thm:elliptic_decomposition}, there exists a FED of order $d=2m-2=4$ in both equilibria. We calculate
    \begin{align*}
        \mathcal{E}(F,a_1)=\mathcal{E}(F,a_2)=\left\{0,\frac{\pi}{2},\pi,\frac{3\pi}{2}\right\}\text{.}
    \end{align*}
    
    \begin{figure}[ht]
        \centering
        \includegraphics[trim={22pt 1pt 40pt 23pt},clip,scale=0.82]{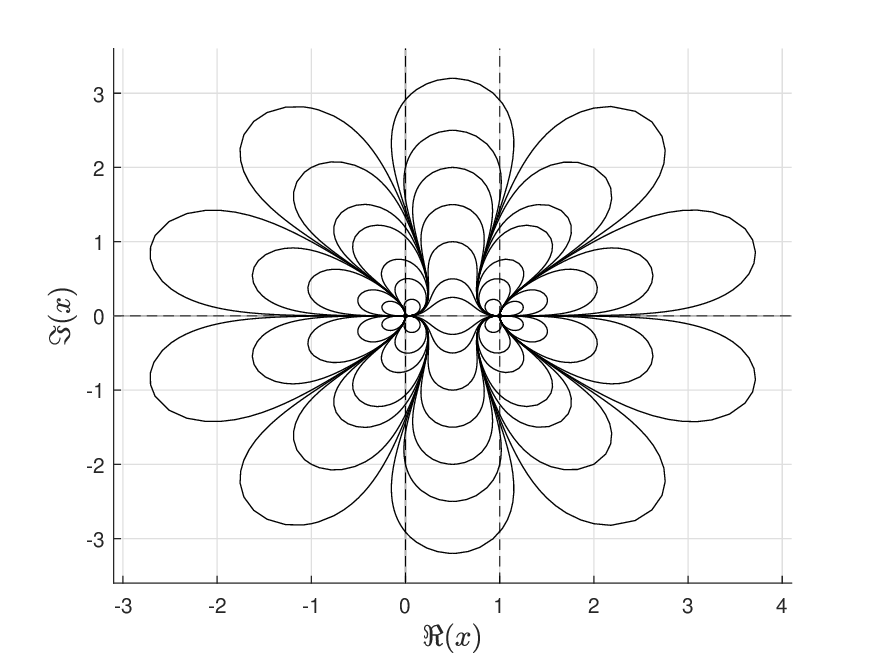}
        \caption{Local phase portrait of system \eqref{eq:planarODE} with $F(x)=x^3(x-1)^3$, plotted with Matlab.}
    \end{figure}

    In this example, the area between the two equilibria appears in some sense not intuitive. Upon closer inspection, it becomes apparent that the orbits "between" the two equilibria are heteroclinic. Thus, we recognise a global behaviour of the orbits, different from that described in Definition \ref{def:sector} c). But this time, the orbits are exclusively bounded. We conclude that also in this example the local structure cannot be transferred to the global phase portrait.
    
    The question now also arises as to whether this different global structure with unbounded orbits (as in Example \ref{ex:exp}) is possible at all in the case of polynomial holomorphic flows. Moreover, how many unbounded orbits can exist in an polynomial holomorphic dynamical system? These questions will be discussed in an upcoming paper.
\end{example}

\section{\bf A Poincaré-Bendixson-type theorem for holomorphic flows}
Using our results in the last section, we formulate and prove a specific Poincaré-Bendixson-type result.
\begin{theorem}[Poincaré-Bendixson for Holomorphic Flows]\label{thm:poincare_bendixson}
    Let $\Omega\subset\C{}$ be a simply connected domain, $F\in\Hol{}(\Omega)$, $F\not\equiv 0$, $K\subset\Omega$ compact and $\xi\in K$ such that $\Gamma_{+(-)}(\xi)\subset K$. Then either $\Gamma(\xi)$ is a periodic orbit with exactly one equilibrium, a center, in its interior, or the positive (negative) limit set of $\Gamma(\xi)$ consists of exactly one equilibrium. Additionally, if $\Gamma(\xi)$ is periodic, the interior of $\Gamma(\xi)$ (except of the center) is filled entirely with periodic orbits all having the center in its interior.
\end{theorem}
\begin{proof}
    Assume w.l.o.g. $\Gamma_+(\xi)\subset K$ and set $\Gamma:=\Gamma_+(\xi)$. By \cite[Theorem 3.2]{broughan2003holomorphic} and \cite[Theorem 4.34]{masterthesiskainz}, respectively, there are no limit cycles in $\Omega$. Hence we can apply the Poincaré-Bendixson Theorem, cf. \cite[3.7, Theorem 1]{perko2013differential}, to conclude that $\Gamma$ is either a periodic solution, or $w_+(\Gamma)$ consists of at least one equilibrium.
    
	First, assume that $\Gamma$ is periodic, i.e. $\Gamma=\Gamma(\xi)$. Since $\Omega$ is simply connected, we have $\overline{\Int(\Gamma)}\subset\Omega$. By applying the results in \cite[\S11]{andronov1973qualitativetheory} as well as Corollary \ref{cor:index_1_center_foci_nodes}, there exists exactly one equilibrium $a\in\Int(\Gamma)$ and this equilibrium is a center, node or focus. Suppose that $a$ is not a center, i.e. $a$ is asymptotically stable in exactly one time direction, say w.l.o.g. $t\to\infty$. Thus, there exists $x\in\Int(\Gamma)\setminus\{a\}$ with $a\in w_+(\Gamma(x))$ and $a\not\in w_-(\Gamma(x))$, i.e. the orbit $\Gamma(x)$ cannot be periodic. Since $a$ is the only equilibrium in $\Int(\Gamma)$, the Generalized Poincaré-Bendixson Theorem, cf. \cite[3.7, Theorem 2]{perko2013differential}, implies the existence of a limit cycle in the limit set $w_-(\Gamma(x))$. This is a contradiction to \cite[Theorem 3.2]{broughan2003holomorphic}. Hence $a$ must be a center. In addition, again by the Generalized Poincare-Bendixson Theorem, all orbits in $\Int(\Gamma)\setminus\{a\}$ must be periodic. If one of these periodic orbits in $\Int(\Gamma)$ had not $a$ in its interior, this would be a contradiction to \cite[\S11, Corollary 1]{andronov1973qualitativetheory}, since $a$ is the only equilibrium in $\Int(\Gamma)$. Hence the set $\Int(\Gamma)\setminus\{a\}$ has to be filled entirely with periodic orbits all having $a$ in its interior.
 
	Secondly, assume that $\Gamma$ is not periodic and $w_+(\Gamma)$ consists of more than exactly one equilibrium. Then, by \cite[3.7, Theorem 2]{perko2013differential}, the following case occurs: The set $w_+(\Gamma)$ consists of a finite number of equilibria together with heteroclinic and homoclinic orbits each having one of these equilibria in their limit sets. We prove that this cannot occur.
 
	Suppose there is at least one homoclinic limit orbit $S\subset w_+(\Gamma)$ in the equilibrium $a\in w_+(\Gamma)$, i.e. we have $w_+(S)=w_-(S)=\{a\}$. Since a center is geometrically impossible and nodes and foci are asymptotically stable in exactly one time direction, $a$ must have a FED, cf. Theorem \ref{thm:elliptic_decomposition}. Moreover, for $x\in\Gamma$ and $y\in S$ sufficiently close to $a$ there exist sequences $(t_k)_{k\in\N{}},(\tilde{t}_k)_{k\in\N{}}\subset\R{}$ with $t_k,\tilde{t}_k\to\infty$, $\Phi(t_k,x)\to a$ and $\Phi(\tilde{t}_k,x)\to y$ for $k\to\infty$. Hence $\Gamma$ comes arbitrarily close to $a$, but does not tend to $a$. This is impossible for such an equilibrium, cf. Definition \ref{def:sector} c) and Proposition \ref{prop:definiteDirections2} (iii).
 
    Suppose there are at least two equilibria $a_1,a_2\in w_+(\Gamma)$, $a_1\not=a_2$, with at least one heteroclinic limit orbit $S\subset w_+(\Gamma)$ satisfying $w_+(S)=\{a_1\}$ and $w_-(S)=\{a_2\}$. Suppose that $a_1$ is a node or focus. Let $\delta>0$ be the radius such that $w_+(\Gamma(y))=\{a_1\}$ for all $y\in\mathcal{B}_\delta(a_1)$, cf. Definition \ref{def:GeometryOfEquilibria}. Since $a_1\in w_+(\Gamma)$, we can choose $y\in\Gamma\cap\mathcal{B}_\delta(a_1)$. We get $a_2\in w_+(\Gamma)=w_+(\Gamma(y))=\{a_1\}$, which is a contradiction. Hence $a_1$ cannot be a node or focus. Since a center is geometrically impossible, $a_1$ must have again a FED, cf. Theorem \ref{thm:elliptic_decomposition}. As in the above case, we can choose suitable time sequences such that $\Gamma$ comes arbitrarily close to $a_1$, but does not tend to $a_1$. This is impossible for $a_1$.
\end{proof}
\begin{corollary}
    \it Let $F\in\Hol{}(\C{})$ be entire and $F\not\equiv 0$. All bounded non-periodic orbits of \eqref{eq:planarODE} are either homoclinic or heteroclinic.
\end{corollary}
\begin{proof}
    This follows directly from Theorem \ref{thm:poincare_bendixson}.
\end{proof}

\bigskip

\bibliographystyle{plain}

\end{document}